\documentclass[12pt,a4paper]{amsart}
\usepackage{amsmath}
\usepackage{amsfonts}
\usepackage{amssymb,latexsym}
\usepackage{enumerate}
\usepackage{hyperref}
\usepackage{hyphenat}
\usepackage{color}
\usepackage{hhline}

\makeatletter
\@namedef{subjclassname@2010}{  \textup{2010} Mathematics Subject Classification}
\makeatother
\newtheorem{theorem}{Theorem}[section]
\newtheorem{corollary}[theorem]{Corollary}
\newtheorem{lemma}[theorem]{Lemma}

\numberwithin{equation}{section}
\begin{document}
\title[Symmetric brackets of connections with totally skew-symmetric torsion]%
{Symmetric brackets induced by\\
connections with totally skew-symmetric torsion\\
on skew-symmetric algebroids}
\author{Bogdan Balcerzak}

\begin{abstract}
In this note, we discuss symmetric brackets on skew-symmetric algebroids
associated with a metric structure. Given a pseudo-Riemannian metric
structure, we describe symmetric brackets induced by connections with
totally skew-symmetric torsion in the language of Lie derivatives and
differentials of functions. In particular, we obtain an explicit formula of
the Levi-Civita connection. We also present some symmetric brackets on
almost Hermitian manifolds. Especially, we discuss the first canonical
Hermitian connection and show the formula for it in the case of nearly K\"{a}%
hler manifolds using the properties of symmetric brackets.
\end{abstract}

\keywords{Skew-symmetric algebroid; almost Lie algebroid; anchored vector
bundle; connection; symmetric product; the symmetrized covariant derivative;
symmetric Lie derivative; connection with totally skew-symmetric torsion\\
\ \\
(\textit{B. Balcerzak}) Institute of Mathematics, Lodz University of
Technology, W\'{o}lcza\'{n}ska 215, 90-924 \L \'{o}d\'{z}, Poland; e-mail:
bogdan.balcerzak@p.lodz.pl}
\maketitle

\section{Introduction}

An \textit{anchored vector bundle} $\left( A,\varrho _{A}\right) $ over a
manifold $M$ is a vector bundle $A$ over $M$ equipped with a homomorphism of
vector bundles $\varrho _{A}:A\rightarrow TM$ over the identity, which is
called an \textit{anchor}. If, additionally, in the space $\Gamma (A)$ of
smooth sections of $A$ we have $\mathbb{R}$-bilinear skew-symmetric mapping $%
\left[ \cdot ,\cdot \right] :\Gamma (A)\times \Gamma (A)\rightarrow \Gamma
(A)$ associated with the anchor with the following derivation law%
\begin{equation}
\left[ X,f\cdot Y\right] =f\cdot \left[ X,Y\right] +(\varrho _{A}\circ
X)(f)\cdot Y  \label{Leibniz-rule}
\end{equation}%
for $X,Y\in \Gamma (A)$, $f\in C^{\infty }(M)$, we say that $\left( A,\rho
_{A},\left[ \cdot ,\cdot \right] \right) $ is a \textit{skew-symmetric
algebroid} over $M$.

If the anchor preserves $\left[ \cdot ,\cdot \right] $ and the Lie bracket $%
\left[ \cdot ,\cdot \right] _{TM}$ of vector fields on $M$, i.e., $\varrho
_{A}\circ \left[ X,Y\right] =\left[ \varrho _{A}\circ X,\varrho _{A}\circ Y%
\right] _{TM}$ for $X,Y\in \Gamma (A)$, a skew-symmetric algebroid is an 
\textit{almost Lie algebroid}. Any skew-symmetric algebroid in which $\left[
\cdot ,\cdot \right] $ satisfies the Jacobi identity is a Lie algebroid in
the sense of Pradines, who discovered them as infinitesimal parts of
differentiable groupoids \cite{Pradines1967} (for the general theory of Lie
algebroids, we refer to the Mackenzie monographs \cite{Mackenzie1987}, \cite%
{Mackenzie2005}). Thus, Lie algebroids are simultaneous generalizations of
integrable distributions on the one hand, and Lie algebras on the other.

Anchored vector bundles, in particular almost Lie algebroids, were
intensively studied by Marcela Popescu and Paul Popescu, among others, in 
\cite{Popescu1996}, \cite{Popescu-Popescu2001BCP}, \cite%
{Popescu-Popescu2001BCP-2}, \cite{Popescu2004} and recently in \cite%
{Popescu2019}, in which the Chern character for almost Lie algebroids is
considered. However, the concept of skew-symmetric algebroids was introduced
by Kosmann-Schwarzbach and Magri in \cite{Kosmann-Schwarzbach-Magri1990} on
the level of finitely generated projective modules over commutative and
associative algebras with unit and under the name \textit{pre-Lie algebroids}%
. Skew-symmetric algebroids (under the same name pre-Lie algebroids) were
examined by Grabowski and Urba\'{n}ski in \cite{Grabowski-Urbanski1997}, 
\cite{Grabowski-Urbanski1999}, where a concept of general algebroids, which
play an important role in analytical mechanics, also was introduced. Using
general algebroids instead of Lie algebroids, Grabowska, Grabowski, and Urba%
\'{n}ski observed in \cite{Grabowska-Grabowski2008} and \cite%
{Grabowska-Grabowski-Urbanski2006} that one can describe a larger family of
systems, both in the Lagrangian and Hamiltonian formalisms.

In this paper, we use the terminology of skew-symmetric algebroid which
comes from de Le\'{o}n, Marrero, and de Diego in \cite{L-M-N2010}, in which
linear almost Poisson structures (also discussed \cite%
{Grabowski-Urbanski1997}, \cite{Grabowski-Urbanski1999}, \cite%
{Kosmann-Schwarzbach-Magri1990}) are applied to nonholonomic mechanical
systems.

Given an anchored vector bundle, we can associate a connection. Given a
skew-symmetric algebroid, we can associate a connection with a torsion.

An $A$\textit{-connection }in a vector bundle $E\rightarrow M$ is an $%
\mathbb{R}$-bilinear map $\nabla :\Gamma (A)\times \Gamma (E)\rightarrow
\Gamma (E)$ with the following properties: 
\begin{eqnarray*}
\nabla _{f\cdot X}(u) &=&f\cdot \nabla _{X}(u), \\
\nabla _{X}(f\cdot u) &=&f\cdot \nabla _{X}(u)+(\varrho _{A}\circ X)(f)\cdot
u
\end{eqnarray*}%
for any $X,Y\in \Gamma (A),$ $f\in C^{\infty }(M)$,$\ u\in \Gamma (E)$.

The \emph{torsion} of an $A$-connection $\nabla $ in $A$ is the tensor $%
T^{\nabla }\in \Gamma (\mathop{\textstyle \bigwedge }\nolimits^{2}A^{\ast
}\otimes A)$ defined by%
\begin{equation*}
T^{\nabla }(X,Y)=\nabla _{X}Y-\nabla _{Y}X-\left[ X,Y\right]
\end{equation*}%
for $X,Y\in \Gamma (A)$. We say that an $A$-connection is \emph{torsion-free}
if its torsion equals zero.

Let $\mathsf{S}^{k}A^{\ast }$ denote the $k$-th symmetric power of the
bundle $A^{\ast }$. On $\mathsf{S}(A)=\mathop{\textstyle \bigoplus }\limits%
_{k\geq 0}\mathsf{S}^{k}A^{\ast }$, we define the operator $d^{s}:\mathsf{S}%
(A)\rightarrow \mathsf{S}(A)$ as the symmetrized covariant derivative, i.e., 
$d^{s}\eta =\left( k+1\right) \cdot \left( \mathop{\rm Sym}\circ \nabla
\right) \eta \ $for $\eta \in \Gamma (\mathsf{S}^{k}A^{\ast })$. This
operator appeared naturally in the study of generalized gradients on Lie
algebroids in the sense of Stein-Weiss in \cite{Balcerzak-Pierzchalski2013}.
However, the operator $d^{s}$ in the case of tangent bundles was introduced
by Sampson in \cite{Sampson1973}. Next, this operator on tangent bundles was
discussed by several authors when studying the Lichnerowicz-type Laplacian
on symmetric tensors, called the Sampson Laplacian. In particular, we
observe an important contribution to the theory of these operators in the
recent papers of Mike\v{s}, Rovenski, Stepanov, and Tsyganok \cite%
{Stepanov-Tsyganok-Mikes2019}, \cite{Mikes-Rovenski-Stepanov2020}. We also
refer to \cite{Mikes-Rovenski-Stepanov2020} for the extended literature on
the Sampson Laplacian.

For $\eta \in \Gamma (\mathsf{S}^{k}A^{\ast })$, $X_{1},\ldots ,X_{k+1}\in
\Gamma \left( A\right) $ holds%
\begin{eqnarray}
\left( d^{s}\eta \right) \left( X_{1},\ldots ,X_{k+1}\right)  &=&%
\mathop{\textstyle \sum }\limits_{j=1}^{k+1}(\rho _{A}\circ X_{j})\hspace{%
-0.1cm}\left( \eta (X_{1},\ldots \widehat{X}_{j}\ldots ,X_{k+1})\right) 
\label{KoszulForm} \\
&&-\mathop{\textstyle \sum }\limits_{i<j}\eta \hspace{-0.05cm}\left(
\left\langle X_{i}:X_{j}\right\rangle ,X_{1},\ldots \widehat{X}_{i}\ldots 
\widehat{X}_{j}\ldots ,X_{k+1}\right) ,  \notag
\end{eqnarray}%
where%
\begin{equation*}
\left\langle X:Y\right\rangle =\nabla _{X}Y+\nabla _{Y}X
\end{equation*}%
for $X,Y\in \Gamma \left( A\right) $. Thus, $d^{s}$ can be written in the
Koszul-type form (\ref{KoszulForm}), where $\left\langle \cdot :\cdot
\right\rangle $ is the symmetric bracket giving on $A$ a pseudo-Lie
algebroid structure in the sense of \cite{Grabowski-Urbanski1997} with $\rho
_{A}$ and $-\rho _{A}$ as the left anchor and the right anchor,
respectively. This Koszul-type shape of $d^{s}$ in the case of tangent
bundles was first obtained by Heydari, Boroojerdian, and Peyghan in \cite%
{H-B-P2006}.

Therefore, given an $A$-connection $\nabla $ in $A$, we can associate a
symmetric $\mathbb{R}$-bilinear mapping $\left\langle \cdot :\cdot
\right\rangle :\Gamma (A)\times \Gamma (A)\rightarrow \Gamma (A)$, defined
for $X,Y\in \Gamma (A)$, by%
\begin{equation*}
\left\langle X:Y\right\rangle =\nabla _{X}Y+\nabla _{Y}X.
\end{equation*}%
We say that this mapping is a \emph{symmetric product} induced by $\nabla $.
The symmetric product in the case of tangent bundles was introduced by
Crouch \cite{Crouch1981}.

Observe that the symmetric product $\left\langle \cdot :\cdot \right\rangle $
satisfies the Leibniz rules and%
\begin{equation*}
\nabla _{X}Y={\textstyle{\frac{1}{2}}}\left( \left[ X,Y\right] +\left\langle
X:Y\right\rangle \right) +{\textstyle{\frac{1}{2}}}T^{\nabla }(X,Y)
\end{equation*}%
for $X,Y\in \Gamma \left( A\right) $. Thus, the symmetric product induced by 
$\nabla $ is a summand of the connection. Our first purpose is to determine
the symmetric products for connections related to the pseudometric
structure. We give an explicit formula for a metric connection with totally
skew-symmetric torsion (but not necessary torsion-free in general) using the
language of symmetric product. To describe these symmetric brackets, we use
the Lie derivative and the exterior derivative operator induced by the
structure of the skew-symmetric algebroid and their symmetric counterparts.
We show that the condition for connections with totally skew-symmetric
torsion to be compatible with the metric is that the (alternating) Lie
derivative of the metric is equal to the minus of the symmetric Lie
derivative of the metric.

We also consider an almost Hermitian structure and some symmetric brackets
associated with connections that are compatible with the metric structure
and the almost complex structure. We consider two structures of the
skew-symmetric algebroid in the almost Hermitian manifold $(M,g,J)$. The
first structure is the tangent bundle with the identity as an anchor and
with the Lie bracket of vector fields. The second skew-symmetric algebroid
structure induced by the almost complex structure $J$, where $J$ is the
anchor and the bracket is associated with the Nijenhuis tensor, was
introduced in \cite{Kosmann-Schwarzbach-Magri1990}. We also discuss the
first canonical Hermitian connection $\overline{\nabla }$ and obtain a
formula for $\overline{\nabla }$ in the case of nearly K\"{a}hler manifolds
using the properties of symmetric brackets.

\section{The exterior derivative operator and the symmetrized covariant
derivative}

Let $\left( A,\varrho _{A},[\cdot ,\cdot ]\right) $ be a skew-symmetric
algebroid over a manifold $M$. The \emph{substitution} operator $%
i_{X}:\Gamma (\mathop{\textstyle \bigotimes }\nolimits^{k}A^{\ast
})\rightarrow \Gamma (\mathop{\textstyle \bigotimes }\nolimits^{k-1}A^{\ast
})$ for $X\in \Gamma (A)$ is defined by%
\begin{equation*}
(i_{X}\zeta )(X_{1},\ldots ,X_{k-1})=\zeta (X,X_{1},\ldots ,X_{k-1})
\end{equation*}%
for $\zeta \in \Gamma (\mathop{\textstyle \bigotimes }\nolimits^{k}A^{\ast })
$, $X,X_{1},\ldots ,X_{k-1}\in \Gamma (A)$.

The (\emph{alternating})\emph{\ Lie derivative} $\mathcal{L}_{X}^{a}:\Gamma (%
\mathop{\textstyle \bigotimes }\nolimits^{k}A^{\ast })\rightarrow \Gamma (%
\mathop{\textstyle \bigotimes }\nolimits^{k}A^{\ast })$ for $X\in \Gamma (A)$
is defined by%
\begin{equation*}
\left( \mathcal{L}_{X}^{a}\Omega \right) (X_{1},\ldots ,X_{k})=(\varrho
_{A}\circ X)(\Omega (X_{1},\ldots ,X_{k}))-\mathop{\textstyle \sum }\limits%
_{i=1}^{k}\Omega (X_{1},\ldots ,\left[ X,X_{i}\right] ,\ldots ,X_{k})
\end{equation*}%
for $\Omega \in \Gamma (\mathop{\textstyle \bigotimes }\nolimits^{k}A^{\ast
})$, $X_{1},\ldots ,X_{k}\in \Gamma (A)$. Notice that $\mathcal{L}%
_{X}^{a}(\eta )\in \Gamma (\mathop{\textstyle \bigwedge }A^{\ast })$ if $%
\eta \in \Gamma (\mathop{\textstyle \bigwedge }A^{\ast })$.

Moreover, let%
\begin{equation*}
\nabla :\Gamma (A)\times \Gamma (A)\rightarrow \Gamma (A)
\end{equation*}%
be an $A$-connection in $A$. We define the $A$-connection $\nabla ^{\ast }$
in the dual bundle in a classical way by the following formula 
\begin{equation*}
(\nabla _{X}^{\ast }\omega )Y=(\varrho _{A}\circ X)(\omega (Y))-\omega
(\nabla _{X}Y)
\end{equation*}%
for $\omega \in \Gamma (A^{\ast })$, $X,Y\in \Gamma (A)$. Next, by the
Leibniz rule, we extend this connection to the $A$-connection in the whole
tensor bundle $\mathop{\displaystyle \bigotimes }A^{\ast }$, which will also
be denoted by $\nabla $. Then for $\zeta \in \Gamma \left( %
\mathop{\displaystyle \bigotimes }^{k}A^{\ast }\right) $, $X,X_{1},\ldots
,X_{k}\in \Gamma (A)$,%
\begin{equation*}
(\nabla _{X}\zeta )(X_{1},\ldots ,X_{k})=(\varrho _{A}\circ X)(\zeta
(X_{1},\ldots ,X_{k}))-\mathop{\displaystyle \sum }\limits_{j=1}^{k}\zeta
(X_{1},\ldots ,\nabla _{X}X_{j},\ldots ,X_{k}).
\end{equation*}%
Now, we define the operator 
\begin{equation*}
\nabla :\Gamma \left( \bigotimes\nolimits^{k}A^{\ast }\right) \rightarrow
\Gamma \left( \bigotimes\nolimits^{k+1}A^{\ast }\right) 
\end{equation*}%
by%
\begin{equation*}
\left( \nabla \zeta \right) \left( X_{1},X_{2}\ldots ,X_{k+1}\right) =\left(
\nabla _{X_{1}}\zeta \right) \left( X_{2},\ldots ,X_{k+1}\right) .
\end{equation*}

We recall that the exterior derivative operator on the skew-symmetric
algebroid $\left( A,\varrho _{A},[\cdot ,\cdot ]\right) $ is defined by%
\begin{eqnarray*}
\left( d^{a}\eta \right) \left( X_{1},\ldots ,X_{k+1}\right)  &=&%
\mathop{\textstyle \sum }\limits_{j=1}^{k+1}\left( -1\right) ^{j+1}(\rho
_{A}\circ X_{j})\hspace{-0.1cm}\left( \eta (X_{1},\ldots \widehat{X}%
_{j}\ldots ,X_{k+1})\right)  \\
&&+\mathop{\textstyle \sum }\limits_{i<j}\left( -1\right) ^{i+j}\eta \hspace{%
-0.1cm}\left( \left[ X_{i},X_{j}\right] ,X_{1},\ldots \widehat{X}_{i}\ldots 
\widehat{X}_{j}\ldots ,X_{k+1}\right) 
\end{eqnarray*}%
for $\eta \in \Gamma (\mathop{\textstyle \bigwedge }\nolimits^{k}A^{\ast })$%
, $X_{1},\ldots ,X_{k+1}\in \Gamma (A)$. Associated with a skew-symmetric
algebroid $\left( A,\varrho _{A},[\cdot ,\cdot ]\right) $ is the \emph{%
Jacobiator} $\mathop{\rm Jac}\nolimits_{[\cdot ,\cdot ]}:\Gamma (A)\times
\Gamma (A)\rightarrow \Gamma (A)$ of the bracket $[\cdot ,\cdot ]$ given by 
\begin{equation*}
\mathop{\rm Jac}\nolimits_{[\cdot ,\cdot
]}(X,Y,Z)=[[X,Y],Z]+[[Z,X],Y]+[[Y,Z],X]
\end{equation*}%
for $X,Y,Z\in \Gamma (A)$. If the bracket $\left[ \cdot ,\cdot \right] $
satisfies the Jacobi identity, i.e., $\mathop{\rm Jac}\nolimits_{[\cdot
,\cdot ]}=0$, $d^{a}\circ d^{a}=0$ (discussed in \cite{Maxim-Raileanu1976}).
If $\nabla $ is torsion-free $A$-connection in $A$, then $d^{a}$ can be
written as the alternation of the operator $\nabla $ (cf. \cite%
{Balcerzak-Pierzchalski2013}), i.e., $d^{a}=\left( k+1\right) \cdot \left( %
\mathop{\rm Alt}\circ \nabla \right) $\ on\ $\Gamma (%
\mathop{\textstyle
\bigwedge }\nolimits^{k}A^{\ast })$, where $\mathop{\rm Alt}$ is the \emph{%
alternator} given by $(\mathop{\rm Alt}\zeta )\left( X_{1},\ldots
,X_{k}\right) =\frac{1}{k!}\mathop{\textstyle \sum }\limits_{\sigma \in
S_{k}}\mathop{\rm sgn}\nolimits\sigma ~\zeta \hspace{-0.1cm}\left( X_{\sigma
\left( 1\right) },\ldots ,X_{\sigma \left( k\right) }\right) $ \ for\ \ $%
\zeta \in \Gamma (\mathop{\textstyle \bigotimes }\nolimits^{k}A^{\ast })$.
Equivalently,%
\begin{equation*}
\left( d^{a}\eta \right) \left( X_{1},\ldots ,X_{k+1}\right) =%
\mathop{\textstyle \sum }\limits_{j=1}^{k+1}\left( -1\right) ^{j+1}\left(
\nabla _{X_{j}}\eta \right) \hspace{-0.1cm}\left( X_{1},\ldots \widehat{X}%
_{j}\ldots ,X_{k+1}\right) 
\end{equation*}%
for $\eta \in \Gamma (\mathop{\textstyle \bigwedge }\nolimits^{k}A^{\ast })$%
, $X_{1},\ldots ,X_{k+1}\in \Gamma (A)$.

Here, we recall the classical Cartan's formulas:

\begin{lemma}
\label{Alternating_Cartan}For any $X,Y\in \Gamma (A)$,

\begin{itemize}
\item[(a)] $\mathcal{L}_{X}^{a}=i_{X}d^{a}+d^{a}i_{X}$ and

\item[(b)] $\mathcal{L}_{X}^{a}i_{Y}-i_{Y}\mathcal{L}_{X}^{a}=i_{\left[ X,Y%
\right] }$.
\end{itemize}
\end{lemma}

The symmetrized covariant derivative is the operator%
\begin{equation*}
d^{s}=\left( k+1\right) \cdot \left( \mathop{\rm Sym}\circ \nabla \right)
:\Gamma (\mathsf{S}^{k}A^{\ast })\rightarrow \Gamma (\mathsf{S}^{k+1}A^{\ast
})
\end{equation*}%
being the symmetrization of $\nabla $ up to a constant on the symmetric
power bundle, where $\mathop{\rm Sym}$ is the \emph{symmetrizer} defined by $%
\left( \mathop{\rm Sym}\zeta \right) \left( X_{1},\ldots ,X_{k}\right) =%
\frac{1}{k!}\mathop{\textstyle \sum }\limits_{\sigma \in S_{k}}\zeta \hspace{%
-0.1cm}\left( X_{\sigma \left( 1\right) },\ldots ,X_{\sigma \left( k\right)
}\right) $\ \ for\ \ $\zeta \in \Gamma (\mathop{\textstyle \bigotimes }%
\nolimits^{k}A^{\ast })$. Equivalently,%
\begin{equation}
\left( d^{s}\eta \right) \left( X_{1},\ldots ,X_{k+1}\right) =%
\mathop{\textstyle \sum }\limits_{j=1}^{k+1}\left( \nabla _{X_{j}}\eta
\right) \hspace{-0.1cm}\left( X_{1},\ldots \widehat{X}_{j}\ldots
,X_{k+1}\right) 
\end{equation}%
for $\eta \in \Gamma (\mathsf{S}^{k}A^{\ast })$, $X_{1},\ldots ,X_{k+1}\in
\Gamma \left( A\right) $. We recall that the operator $d^{s}$ in the case of
tangent bundles was introduced by Sampson in \cite{Sampson1973}, in which a
symmetric version of Chern's theorem is proved. This operator on tangent
bundles was discussed in \cite{H-B-P2006}, in which a Fr\"{o}%
licher--Nijenhuis bracket for vector valued symmetric tensors is also
discussed and in \cite{Balcerzak2019}, in which the Dirac-type operator on
symmetric tensors was considered. One can check that for $\eta \in \Gamma (%
\mathsf{S}^{k}A^{\ast })$, $X_{1},\ldots ,X_{k+1}\in \Gamma \left( A\right) $
the following Koszul-type formula holds%
\begin{eqnarray*}
\left( d^{s}\eta \right) \left( X_{1},\ldots ,X_{k+1}\right)  &=&%
\mathop{\textstyle \sum }\limits_{j=1}^{k+1}(\rho _{A}\circ X_{j})\hspace{%
-0.1cm}\left( \eta (X_{1},\ldots \widehat{X}_{j}\ldots ,X_{k+1})\right)  \\
&&-\mathop{\textstyle \sum }\limits_{i<j}\eta \hspace{-0.05cm}\left(
\left\langle X_{i}:X_{j}\right\rangle ,X_{1},\ldots \widehat{X}_{i}\ldots 
\widehat{X}_{j}\ldots ,X_{k+1}\right) ,
\end{eqnarray*}%
where%
\begin{equation*}
\left\langle X:Y\right\rangle =\nabla _{X}Y+\nabla _{Y}X
\end{equation*}%
for $X,Y\in \Gamma \left( A\right) $. This shape of $d^{s}$ in the case $A=TM
$ was discovered by Heydari, Boroojerdian, and Peyghan in \cite{H-B-P2006}.
The symmetric $\mathbb{R}$-bilinear form%
\begin{equation*}
\left\langle \cdot :\cdot \right\rangle :\Gamma \left( A\right) \times
\Gamma \left( A\right) \longrightarrow \Gamma \left( A\right) ,\ \ \
\left\langle X:Y\right\rangle =\nabla _{X}Y+\nabla _{Y}X
\end{equation*}%
is called the \textbf{symmetric product} induced by the $A$-connection $%
\nabla $. The symmetric product in the case $A=TM$ was first introduced by
Crouch in \cite{Crouch1981}. However, the symmetric product for Lie
algebroids was first considered in the context of control systems by Cort%
\'{e}s and Mart\'{\i}nez in \cite{Cortes-Martinez2004}. Observe that 
\begin{equation*}
\left\langle X:f\cdot Y\right\rangle =f\cdot \left\langle X:Y\right\rangle
+\left( \varrho _{A}\circ X\right) (f)\cdot Y
\end{equation*}%
for all $X,Y\in \Gamma \left( A\right) $ and $f\in C^{\infty }(M)$.
Therefore, $\left\langle \cdot :\cdot \right\rangle $ satisfies the
Leibniz-kind rule. Lewis in \cite{Lewis1998} gives some interesting
geometrical interpretation of the symmetric product associated with the
geodesically invariant property of a distribution. We say that a smooth
distribution $D$ on a manifold $M$ with an affine connection $\nabla ^{TM}$
is \textit{geodesically invariant} if for every geodesic $c:I\rightarrow M$
satisfying the property $c^{\prime }(s)\in D_{c(s)}$ for some $s\in I$, we
have $c^{\prime }(s)\in D_{c(s)}$ for every $s\in I$. Lewis proved in \cite%
{Lewis1998} that a distribution $D$ on a manifold $M$ equipped with an
affine connection $\nabla ^{TM}$ is geodesically invariant if and only if
the symmetric product induced by $\nabla ^{TM}$ is closed under $D$.

\section{Symmetric bracket. Symmetric Lie derivative}

Let $\left( A,\varrho _{A},[\cdot ,\cdot ]\right) $ be a skew-symmetric
algebroid over a manifold $M$. A \emph{symmetric bracket} on the anchored
vector bundle $\left( A,\varrho _{A}\right) $ is an $\mathbb{R}$-bilinear
symmetric mapping 
\begin{equation*}
\left\langle \cdot :\cdot \right\rangle :\Gamma (A)\times \Gamma
(A)\rightarrow \Gamma (A)
\end{equation*}%
satisfying the following Leibniz-kind rule:%
\begin{equation*}
\left\langle X:fY\right\rangle =f\left\langle X:Y\right\rangle +(\varrho
_{A}\circ X)(f)Y
\end{equation*}%
for $X,Y\in \Gamma (A)$, $f\in C^{\infty }(M)$.

Let us assume that the skew-symmetric algebroid $\left( A,\varrho
_{A},[\cdot ,\cdot ]\right) $ is equipped with a symmetric bracket $%
\left\langle \cdot :\cdot \right\rangle :\Gamma (A)\times \Gamma
(A)\rightarrow \Gamma (A)$.

We define the operator $d^{s}:\Gamma (\mathsf{S}^{k}A^{\ast })\rightarrow
\Gamma (\mathsf{S}^{k+1}A^{\ast })$ on symmetric power bundle $\mathsf{S}(A)$
for $\eta \in \Gamma (\mathsf{S}^{k}A^{\ast })$, $X_{1},\ldots ,X_{k+1}\in
\Gamma \left( A\right) $ by%
\begin{eqnarray*}
\left( d^{s}\eta \right) \left( X_{1},\ldots ,X_{k+1}\right)  &=&%
\mathop{\textstyle \sum }\limits_{j=1}^{k+1}(\rho _{A}\circ X_{j})\hspace{%
-0.1cm}\left( \eta (X_{1},\ldots \widehat{X}_{j}\ldots ,X_{k+1})\right)  \\
&&-\mathop{\textstyle \sum }\limits_{i<j}\eta \hspace{-0.1cm}\left(
\left\langle X_{i}:X_{j}\right\rangle ,X_{1},\ldots \widehat{X}_{i}\ldots 
\widehat{X}_{j}\ldots ,X_{k+1}\right) .
\end{eqnarray*}%
Next, we extend this operator to the whole tensor bundle by the formula%
\begin{equation*}
d^{s}:\Gamma (\mathop{\textstyle \bigotimes }\nolimits^{k}A^{\ast
})\rightarrow \Gamma (\mathop{\textstyle \bigotimes }\nolimits^{k+1}A^{\ast
}),
\end{equation*}%
\begin{eqnarray*}
\left( d^{s}\Omega \right) \left( X_{1},\ldots ,X_{k+1}\right)  &=&%
\mathop{\textstyle \sum }\limits_{j=1}^{k+1}(\rho _{A}\circ X_{j})\hspace{%
-0.1cm}\left( \Omega (X_{1},\ldots \widehat{X}_{j}\ldots ,X_{k+1})\right)  \\
&&-\mathop{\textstyle \sum }\limits_{i<j}\Omega \hspace{-0.05cm}\left(
X_{1},\ldots \widehat{X}_{i}\ldots ,\left\langle X_{i}:X_{j}\right\rangle
,\ldots ,X_{k+1}\right) 
\end{eqnarray*}%
for $\Omega \in \Gamma (\mathop{\textstyle \bigotimes }\nolimits^{k}A^{\ast
})$, $X_{1},\ldots ,X_{k+1}\in \Gamma (A)$.

The \emph{symmetric Lie derivative} $\mathcal{L}_{X}^{s}:\Gamma (%
\mathop{\textstyle \bigotimes }\nolimits^{k}A^{\ast })\rightarrow \Gamma (%
\mathop{\textstyle \bigotimes }\nolimits^{k}A^{\ast })$ for $X\in \Gamma (A)$
is defined by%
\begin{equation*}
\left( \mathcal{L}_{X}^{s}\Omega \right) (X_{1},\ldots ,X_{k})=(\varrho
_{A}\circ X)(\Omega (X_{1},\ldots ,X_{k}))-\mathop{\textstyle \sum }\limits%
_{i=1}^{k}\Omega (X_{1},\ldots ,\left\langle X:X_{i}\right\rangle ,\ldots
,X_{k})
\end{equation*}%
for $\Omega \in \Gamma (\mathop{\textstyle \bigotimes }\nolimits^{k}A^{\ast
})$, $X_{1},\ldots ,X_{k}\in \Gamma (A)$. Notice that the image $\mathcal{L}%
_{X}^{s}(\varphi )$ of a symmetric tensor $\varphi $ is also a symmetric
tensor.

Next, observe that the symmetric Lie derivative satisfies the following
Cartan's identities.

\begin{lemma}
\label{Symmetric_Cartan}For any $X,Y\in \Gamma (A)$,

\begin{itemize}
\item[(a)] $\mathcal{L}_{X}^{s}=i_{X}d^{s}-d^{s}i_{X}$ and

\item[(b)] $\mathcal{L}_{X}^{s}i_{Y}-i_{Y}\mathcal{L}_{X}^{s}=i_{\left%
\langle X:Y\right\rangle }$.
\end{itemize}
\end{lemma}

\begin{proof}
Let $X,Y,X_{1},\ldots ,X_{k}\in \Gamma (A)$ and $\Omega \in \Gamma (%
\mathop{\textstyle \bigotimes }\nolimits^{k}A^{\ast })$.\newline
Observe that%
\begin{eqnarray*}
&&\left( i_{X}d^{s}\Omega \right) (X_{1},\ldots ,X_{k}) \\
&=&(\varrho _{A}\circ X)(\Omega (X_{1},\ldots ,X_{k}))+\mathop{\textstyle
\sum }\limits_{i=1}^{k}(\varrho _{A}\circ X_{i})(\Omega (X,X_{1},\ldots ,%
\widehat{X}_{i},\ldots ,X_{k})) \\
&&-\mathop{\textstyle \sum }\limits_{i=1}^{k}\Omega (X_{1},\ldots
,\left\langle X:X_{i}\right\rangle ,\ldots ,X_{k})-\mathop{\textstyle \sum }%
\limits_{i<j}\Omega (X,X_{1},\ldots ,\widehat{X}_{i},\ldots
,\left\langle X_{i}:X_{j}\right\rangle ,\ldots ,X_{k})
\end{eqnarray*}%
and%
\begin{eqnarray*}
&&\left( d^{s}i_{X}\Omega \right) (X_{1},\ldots ,X_{k}) \\
&=&\mathop{\textstyle \sum }\limits_{i=1}^{k}(\varrho _{A}\circ
X_{i})(\Omega (X,X_{1},\ldots ,\widehat{X}_{i},\ldots ,X_{k}))-%
\mathop{\textstyle \sum }\limits_{i<j}\Omega (X,X_{1},\ldots ,\widehat{X}%
_{i},\ldots ,\left\langle X_{i}:X_{j}\right\rangle ,\ldots ,X_{k}).
\end{eqnarray*}%
Hence, we obtain (a) in Lemma \ref{Symmetric_Cartan}.

Moreover,%
\begin{eqnarray*}
&&\left( \mathcal{L}_{X}^{s}i_{Y}\Omega \right) (X_{1},\ldots ,X_{k-1}) \\
&=&(\varrho _{A}\circ X)(\Omega (Y,X_{1},\ldots ,X_{k-1}))-%
\mathop{\textstyle \sum }\limits_{i=1}^{k-1}\Omega (Y,X_{1},\ldots
,\left\langle X:X_{i}\right\rangle ,\ldots ,X_{k-1})
\end{eqnarray*}%
and%
\begin{multline*}
\left( i_{Y}\mathcal{L}_{X}^{s}\Omega \right) (X_{1},\ldots
,X_{k-1})=(\varrho _{A}\circ X)(\Omega (Y,X_{1},\ldots ,X_{k-1})) \\
-\Omega (\left\langle X:Y\right\rangle ,X_{1},\ldots ,X_{k-1})-%
\mathop{\textstyle \sum }\limits_{i=1}^{k-1}\Omega (Y,X_{1},\ldots
,\left\langle X:X_{i}\right\rangle ,\ldots ,X_{k-1}),
\end{multline*}%
which give immediately (b) in Lemma \ref{Symmetric_Cartan}.
\end{proof}

By the definition of the symmetric Lie derivative and properties of
differentiations, we have the following result:

\begin{lemma}
\label{Lemma_Properties_of_Lie_Sym}For $f\in C^{\infty }(M)$, $X\in \Gamma
(A)$, $\omega \in \Gamma (A^{\ast })$, we have

\begin{itemize}
\item[(a)] $\mathcal{L}_{f\cdot X}^{s}\omega =f\cdot \mathcal{L}%
_{X}^{s}\omega -\left( i_{X}\omega \right) \cdot d^{s}f$ and

\item[(b)] $\mathcal{L}_{X}^{s}(f\cdot \omega )=f\cdot \mathcal{L}%
_{X}^{s}\omega +(\varrho _{A}\circ X)(f)\cdot \omega $.
\end{itemize}
\end{lemma}

\section{The symmetric bracket of a skew-symmetric algebroid with \newline
totally skew-symmetric torsion}

Let $\left( A,\varrho _{A},[\cdot ,\cdot ]\right) $ be a skew-symmetric
algebroid over a manifold $M$ equipped with a pseudo-Riemannian metric $g\in
\Gamma (S^{2}A^{\ast })$ in the vector bundle $A$ and an $A$-connection $%
\nabla $ in $A$. Let $\left\langle \cdot :\cdot \right\rangle $ be the
symmetric product induced by $\nabla $ and $d^{s}$ the symmetrized covariant
derivative. The pseudo-Riemannian metric defines two homomorphisms of vector
bundles%
\begin{equation*}
\flat :A\rightarrow A^{\ast },
\end{equation*}%
\begin{equation*}
\sharp :A^{\ast }\longrightarrow A
\end{equation*}%
by 
\begin{equation*}
\flat (X)=i_{X}g
\end{equation*}%
and%
\begin{equation*}
g(\sharp (\omega ),X)=\omega (X)
\end{equation*}%
for\ $X\in \Gamma (A)$,$\ \omega \in \Gamma \left( A^{\ast }\right) $,
respectively. For any $X\in \Gamma (A)$, the $1$-form $i_{X}g=g(X,\cdot )$
will be denoted, briefly, by $X^{\flat }$.

We say that $\nabla $ is a \textit{connection with totally skew-symmetric
torsion} with respect to a pseudo-Riemannian metric $g$ if the tensor $%
T^{g}\in \Gamma \left( \mathop{\displaystyle \bigotimes }\nolimits%
^{3}A^{\ast }\right) $ given by%
\begin{equation*}
T^{g}(X,Y,Z)=g(T^{\nabla }(X,Y),Z)
\end{equation*}%
for $X,Y,Z\in \Gamma (A)$, is a $3$-form on $A$, i.e., $T^{g}\in \Gamma (%
\mathop{\textstyle \bigwedge }\nolimits^{3}A^{\ast })$  (cf. \cite%
{Agricola-Srni}).

\begin{theorem}
\label{Thm_Conection_Formula_One}Let $X,Z\in \Gamma (A)$.\newline
Then%
\begin{eqnarray*}
g(\nabla _{X}X,Z) &=&g(\sharp (\mathcal{L}_{X}^{a}X^{\flat }-{\textstyle {%
\frac{1 }{2}}}d^{a}(g(X,X)),Z) \\
&&-g(T^{\nabla }(X,Z),X) \\
&&+(\nabla g)(Z,X,X)-{\textstyle {\frac{1 }{2}}}(d^{s}g)(X,X,Z).
\end{eqnarray*}%
In particular, if $\nabla $ is a connection with totally skew-symmetric
torsion compatible with $g$, then 
\begin{equation}
\nabla _{X}X=\sharp (\mathcal{L}_{X}^{a}X^{\flat }-{\textstyle {\frac{1 }{2}}%
}d^{a}(g(X,X)).  \label{formula_for_totally_skew}
\end{equation}
\end{theorem}

\begin{proof}
Let $X,Z\in \Gamma (A)$. First, observe that%
\begin{equation*}
(d^{s}g)(X,X,Z)=2(\nabla g)(X,X,Z)+(\nabla g)(Z,X,X).
\end{equation*}%
Therefore, we have%
\begin{equation*}
(\nabla g)(Z,X,X)-{\textstyle{\frac{1}{2}}}(d^{s}g)(X,X,Z)={\textstyle{\frac{%
1}{2}}}(\nabla g)(Z,X,X)-(\nabla g)(X,X,Z).
\end{equation*}%
Next, observe that%
\begin{eqnarray*}
&&{\textstyle{\frac{1}{2}}}(\nabla g)(Z,X,X)-(\nabla g)(X,X,Z) \\
&=&{\textstyle{\frac{1}{2}}}(\nabla _{Z}g)(X,X)-(\nabla _{X}g)(X,Z) \\
&=&{\textstyle{\frac{1}{2}}}\varrho _{A}(Z)(g(X,X))-g(\nabla
_{Z}X,X)-\varrho _{A}(X)(g(X,Z))+g(\nabla _{X}X,Z)+g(X,\nabla _{X}Z) \\
&=&{\textstyle{\frac{1}{2}}}\varrho _{A}(Z)(g(X,X))+g(\nabla _{X}Z-\nabla
_{Z}X-\left[ X,Z\right] ,X) \\
&&-\varrho _{A}(X)(g(X,Z))+g(\left[ X,Z\right] ,X)+g(\nabla _{X}X,Z).
\end{eqnarray*}%
Since 
\begin{equation*}
\varrho _{A}(Z)(g(X,X))=d^{a}(g(X,X))(Z)=g(\sharp (d^{a}(g(X,X))),Z)
\end{equation*}%
and%
\begin{eqnarray*}
\left( \mathcal{L}_{X}^{a}X^{\flat }\right) (Z) &=&\varrho _{A}(X)(X^{\flat
}(Z))-X^{\flat }(\left[ X,Z\right] ) \\
&=&\varrho _{A}(X)(g(X,Z))-g(X,\left[ X,Z\right] ),
\end{eqnarray*}%
we have 
\begin{eqnarray*}
&&{\textstyle{\frac{1}{2}}}(\nabla g)(Z,X,X)-(\nabla g)(X,X,Z) \\
&=&{\textstyle{\frac{1}{2}}}d^{a}(g(X,X))(Z)+g(T^{\nabla }(X,Z),X)-\left( 
\mathcal{L}_{X}^{a}X^{\flat }\right) (Z)+g(\nabla _{X}X,Z).
\end{eqnarray*}%
Moreover, if $\nabla $ is a metric connection with totally skew-symmetric
torsion, then $\nabla g=0$, $d^{s}g=0$, and 
\begin{equation*}
g(T^{\nabla }(X,Z),X)=-g(T^{\nabla }(X,X),Z)=0,
\end{equation*}%
and, in consequence, we obtain (\ref{formula_for_totally_skew}). This
completes the proof.
\end{proof}

Applying Theorem \ref{Thm_Conection_Formula_One}, we have

\begin{theorem}
\label{Thm_Sym_Product_Second}Let $X,Y,Z\in \Gamma (A)$ and let $%
\left\langle X:Y\right\rangle $ be the symmetric bracket of sections induced
by $\nabla $, i.e., $\left\langle X:Y\right\rangle =\nabla _{X}Y+\nabla
_{Y}X $. Then%
\begin{eqnarray}
g(\left\langle X:Y\right\rangle ,Z) &=&g(\sharp (\mathcal{L}_{X}^{a}Y^{\flat
}+\mathcal{L}_{Y}^{a}X^{\flat }-d^{a}(g(X,Y))),Z)  \label{wzor_g_connection}
\\
&&-g(T^{\nabla }(X,Z),Y)-g(T^{\nabla }(Y,Z),X)  \notag \\
&&+2(\nabla g)(Z,X,Y)-(d^{s}g)(X,Y,Z).  \notag
\end{eqnarray}
\end{theorem}

\begin{proof}
Using the following polarization formula%
\begin{equation*}
\left\langle X:Y\right\rangle =\nabla _{X+Y}\left( X+Y\right) -\nabla
_{X}X-\nabla _{Y}Y
\end{equation*}%
and Theorem \ref{Thm_Conection_Formula_One}, we obtain%
\begin{eqnarray*}
\left\langle X:Y\right\rangle  &=&g(\sharp (\mathcal{L}_{X+Y}^{a}\left(
X+Y\right) ^{\flat }-{\textstyle{\frac{1}{2}}}d^{a}(g(X+Y,X+Y)),Z)-g(T^{%
\nabla }(X+Y,Z),X+Y) \\
&&+(\nabla g)(Z,X+Y,X+Y)-{\textstyle{\frac{1}{2}}}(d^{s}g)(X+Y,X+Y,Z) \\
&&-g(\sharp (\mathcal{L}_{X}^{a}X^{\flat }-{\textstyle{\frac{1}{2}}}%
d^{a}(g(X,X)),Z) \\
&&+g(T^{\nabla }(X,Z),X)-(\nabla g)(Z,X,X)+{\textstyle{\frac{1}{2}}}%
(d^{s}g)(X,X,Z) \\
&&-g(\sharp (\mathcal{L}_{Y}^{a}Y^{\flat }-{\textstyle{\frac{1}{2}}}%
d^{a}(g(Y,Y)) \\
&&+g(T^{\nabla }(Y,Z),Y)-(\nabla g)(Z,Y,Y)+{\textstyle{\frac{1}{2}}}%
(d^{s}g)(Y,Y,Z).
\end{eqnarray*}%
First observe that%
\begin{equation*}
\mathcal{L}_{X+Y}^{a}\left( X+Y\right) ^{\flat }-\mathcal{L}_{X}^{a}X^{\flat
}-\mathcal{L}_{Y}^{a}Y^{\flat }=\mathcal{L}_{X}^{a}Y^{\flat }+\mathcal{L}%
_{Y}^{a}X^{\flat }
\end{equation*}%
and%
\begin{equation*}
-{\textstyle{\frac{1}{2}}}d^{a}(g(X+Y,X+Y)+{\textstyle{\frac{1}{2}}}%
d^{a}(g(X,X))+{\textstyle{\frac{1}{2}}}d^{a}(g(Y,Y)=-d^{a}(g(X,Y)).
\end{equation*}%
Since $g$ is a symmetric tensor and $T^{\nabla }$ is skew-symmetric, we
conclude that%
\begin{equation*}
-g(T^{\nabla }(X+Y,Z),X+Y)+g(T^{\nabla }(X,Z),X)+g(T^{\nabla }(Y,Z),Y)
\end{equation*}%
is equal to%
\begin{equation*}
-g(T^{\nabla }(X,Z),Y)-g(T^{\nabla }(Y,Z),X).
\end{equation*}%
Moreover,%
\begin{equation*}
(\nabla g)(Z,X+Y,X+Y)-(\nabla g)(Z,X,X)-(\nabla g)(Z,Y,Y)=2(\nabla g)(Z,X,Y)
\end{equation*}%
and%
\begin{align*}
(d^{s}g)(X,Y,Z)& ={\textstyle{\frac{1}{2}}}(d^{s}g)(X,Y,Z)+{\textstyle{\frac{%
1}{2}}}(d^{s}g)(Y,X,Z) \\
& ={\textstyle{\frac{1}{2}}}(d^{s}g)(X+Y,X+Y,Z)-{\textstyle{\frac{1}{2}}}%
(d^{s}g)(X,X,Z)-{\textstyle{\frac{1}{2}}}(d^{s}g)(Y,Y,Z).
\end{align*}%
Hence, it is clear that some summands of $\left\langle X:Y\right\rangle $
cancel. This establishes the formula (\ref{wzor_g_connection}).
\end{proof}

The formula in Theorem \ref{Thm_Sym_Product_Second} gives an explicit
formula of symmetric bracket defined by any metric connection with totally
skew-symmetric torsion.

\begin{corollary}
Let $\nabla $ be any metric $A$-connection in $A$ with totally
skew-symmetric torsion with respect to a pseudo-Riemannian metric $g$. Then%
\begin{equation*}
\nabla _{X}Y+\nabla _{Y}X=\sharp (\mathcal{L}_{X}^{a}Y^{\flat }+\mathcal{L}%
_{Y}^{a}X^{\flat }-d^{a}(g(X,Y)).
\end{equation*}
\end{corollary}

\section{A general metric compatibility condition of connections having
totally skew-symmetric torsion. The Levi-Civita connection}

Let $\left( A,\varrho _{A},[\cdot ,\cdot ]\right) $ be a skew-symmetric
algebroid over a manifold $M$ equipped with a pseudo-Riemannian metric $g\in
\Gamma (S^{2}A^{\ast })$ in the vector bundle $A$ and a symmetric bracket $%
\left\langle \cdot :\cdot \right\rangle :\Gamma (A)\times \Gamma
(A)\rightarrow \Gamma (A)$. By definition, we recall that the symmetric
bracket is an $\mathbb{R}$-bilinear symmetric mapping which satisfies the
following Leibniz-kind rule:%
\begin{equation*}
\left\langle X:fY\right\rangle =f\left\langle X:Y\right\rangle +(\varrho
_{A}\circ X)(f)\cdot Y
\end{equation*}%
for $X,Y\in \Gamma (A)$, $f\in C^{\infty }(M)$.

Given the bundle metric $g$ on $A$, there is a unique $A$-connection in $A$
which is torsion-free and metric-compatible (i.e., $T^{\nabla }=0$ and $%
\nabla g=0$). We call such an $A$-connection the \textit{Levi-Civita
connection} with respect to $g$. Let $\mathcal{L}^{s}$ and $d^{s}$ denote
the symmetric Lie derivative and the symmetric derivative operator,
respectively, and both are induced by $\left\langle \cdot :\cdot
\right\rangle $.

\begin{theorem}
\label{Theorem_Comptibility_by_Lie_der}Let $\nabla $ be an $A$-connection in 
$A$ with totally skew-symmetric torsion with respect to a pseudo-Riemannian
metric $g$ on $A$ given by 
\begin{equation}
\nabla _{X}Y={\textstyle {\frac{1 }{2}}}\left( \left[ X,Y\right]
+\left\langle X:Y\right\rangle \right) +{\textstyle {\frac{1 }{2}}}T(X,Y)
\label{Formula_totally_skew_sym}
\end{equation}%
for $X,Y\in \Gamma (A)$. Then%
\begin{equation*}
(i_{X}\circ \nabla )g={\textstyle {\frac{1 }{2}}}\left( \mathcal{L}_{X}^{a}+%
\mathcal{L}_{X}^{s}\right) g
\end{equation*}%
for$\ X\in \Gamma (A)$.
\end{theorem}

\begin{proof}
Let $X,Y,Z\in \Gamma (A)$. Since $T\in \Gamma \left( 
\mathop{\displaystyle
\bigwedge }\nolimits^{2}A^{\ast }\otimes A \right) $ is a $2$-skew-symmetric
tensor with the property that%
\begin{equation*}
g(Y,T(X,Z))=g(T(X,Z),Y)=-g(T(X,Y),Z),
\end{equation*}%
we have,%
\begin{eqnarray*}
\left( \nabla _{X}g\right) (Y,Z) &=&\rho _{A}(X)(g(Y,Z))-g(\nabla
_{X}Y,Z)-g(Y,\nabla _{X}Z) \\
&=&{\textstyle {\frac{1 }{2}}}\left( \rho
_{A}(X)(g(Y,Z))-g([X,Y],Z)-g(Y,[X,Z])\right) \\
&&+{\textstyle {\frac{1 }{2}}}\left( \rho _{A}(X)(g(Y,Z))-g(\left\langle
X:Y\right\rangle ,Z)-g(Y,\left\langle X:Z\right\rangle )\right) \\
&&-{\textstyle {\frac{1 }{2}}}g(T(X,Y),Z)-{\textstyle {\frac{1 }{2}}}%
g(Y,T(X,Z)) \\
&=&{\textstyle {\frac{1 }{2}}}\left( \mathcal{L}_{X}^{a}g+\mathcal{L}%
_{X}^{s}g\right) (Y,Z)+0.
\end{eqnarray*}
\end{proof}

Hence, we can conclude the following condition on a connection with totally
skew-symmetric torsion to be a metric connection:

\begin{corollary}
If $\nabla $ is an $A$-connection with totally skew-symmetric torsion with
respect to $g$ given by \emph{(\ref{Formula_totally_skew_sym})}, then $%
\nabla $ is metric with respect to $g$ if and only if%
\begin{equation*}
\mathcal{L}_{X}^{a}g=-\mathcal{L}_{X}^{s}g
\end{equation*}%
for any $X\in \Gamma (A)$.
\end{corollary}

Now, we recall some properties of the (skew-symmetric) Lie derivative.

\begin{lemma}
\label{Lemma_Skew_symm_Lie_derivative}For $f\in C^{\infty }(M)$, $X\in
\Gamma (A)$, $\omega \in \Gamma (A^{\ast })$, we have

\begin{itemize}
\item[(a)] $\mathcal{L}_{f\cdot X}^{a}\omega =f\cdot \mathcal{L}%
_{X}^{a}\omega +\left( i_{X}\omega \right) \cdot d^{a}f$ and

\item[(b)] $\mathcal{L}_{X}^{a}(f\cdot \omega )=f\cdot \mathcal{L}%
_{X}^{a}\omega +(\varrho _{A}\circ X)(f)\cdot \omega $.
\end{itemize}
\end{lemma}

\begin{theorem}
\label{Theorem_Sym_Bracket_by_Lie_der}Given a skew-symmetric algebroid $%
\left( A,\varrho _{A},[\cdot ,\cdot ]\right) $, we define 
\begin{equation*}
\left\langle X:Y\right\rangle ^{s}:\Gamma (A)\times \Gamma (A)\rightarrow
\Gamma (A)
\end{equation*}%
by%
\begin{equation}
\left\langle X:Y\right\rangle ^{s}=\sharp (\mathcal{L}_{X}^{a}Y^{\flat }+%
\mathcal{L}_{Y}^{a}X^{\flat }-d^{a}(g(X,Y))
\label{formula_for_symetric_bracket}
\end{equation}%
for $X,Y\in \Gamma (A)$. Then, $\left\langle \cdot :\cdot \right\rangle ^{s}$
is a symmetric bracket that defines the symmetric Lie derivative $\mathcal{L}%
^{s}$ satisfying $\mathcal{L}_{X}^{s}g=-\mathcal{L}_{X}^{a}g$.
\end{theorem}

\begin{proof}
It is evident that $\left\langle \cdot :\cdot \right\rangle ^{s}$ is a
symmetric and $\mathbb{R}$-bilinear mapping. Let $X,Y,Z\in \Gamma (A)$.
Lemma \ref{Lemma_Skew_symm_Lie_derivative} now gives%
\begin{equation*}
\mathcal{L}_{X}^{a}(fY)^{\flat }=f\mathcal{L}_{X}^{a}Y^{\flat }+(\varrho
_{A}\circ X)(f)Y^{\flat }
\end{equation*}%
and%
\begin{equation*}
\mathcal{L}_{fY}^{a}X^{\flat }=f\mathcal{L}_{Y}^{a}X^{\flat }+g(X,Y)d^{a}f.
\end{equation*}%
Since%
\begin{equation*}
d^{a}(g(X,fY))=fd^{a}(g(X,Y))+g(X,Y)d^{a}f,
\end{equation*}%
we conclude that $\left\langle \cdot :\cdot \right\rangle ^{s}$ satisfies
the Leibniz rule. In consequence, $\left\langle \cdot :\cdot \right\rangle
^{s}$ is a symmetric bracket. Observe that%
\begin{eqnarray*}
g(\left\langle X:Y\right\rangle ^{s},Z) &=&\left( \left\langle
X:Y\right\rangle ^{s}\right) ^{\flat }(Z)=(\mathcal{L}_{X}^{a}Y^{\flat }+%
\mathcal{L}_{Y}^{a}X^{\flat }-d^{a}(g(X,Y))(Z) \\
&=&(\varrho _{A}\circ X)(g(Y,Z))-g(Y,\left[ X,Z\right] ) \\
&&+(\varrho _{A}\circ Y)(g(X,Z))-g(X,\left[ Y,Z\right] )-(\varrho _{A}\circ
Z)(g(X,Y)).
\end{eqnarray*}%
Similarly,%
\begin{eqnarray*}
g(Y,\left\langle X:Z\right\rangle ^{s}) &=&(\varrho _{A}\circ
X)(g(Y,Z))-g(Z, \left[ X,Y\right] ) \\
&&+(\varrho _{A}\circ Z)(g(X,Y))-g(X,\left[ Z,Y\right] )-(\varrho _{A}\circ
Y)(g(X,Z)).
\end{eqnarray*}%
Therefore,%
\begin{eqnarray*}
(\mathcal{L}_{X}^{s}g)(Y,Z) &=&(\varrho _{A}\circ X)(g(Y,Z))-g(\left\langle
X:Y\right\rangle ^{s},Z)-g(Y,\left\langle X:Z\right\rangle ^{s}) \\
&=&g(Y,\left[ X,Z\right] )+g(X,\left[ Y,Z\right] ) \\
&&-(\varrho _{A}\circ X)(g(Y,Z))+g(Z,\left[ X,Y\right] )+g(X,\left[ Z,Y%
\right] ) \\
&=&-(\varrho _{A}\circ X)(g(Y,Z))+g(\left[ X,Y\right] ,Z)+g(Y,\left[ X,Z%
\right] ) \\
&&+g(X,\left[ Y,Z\right] +[Z,Y]) \\
&=&-(\mathcal{L}_{X}^{a}g)(Y,Z)+0.
\end{eqnarray*}
\end{proof}

Theorem \ref{Theorem_Comptibility_by_Lie_der} now yields

\begin{corollary}
If $\nabla $ is an $A$-connection in the bundle $A$ with totally
skew-symmetric torsion defined, for $X,Y\in \Gamma (A)$, by%
\begin{equation*}
\nabla _{X}Y={\textstyle {\frac{1 }{2}}}\left( \left[ X,Y\right]
+\left\langle X:Y\right\rangle ^{s}\right) +{\textstyle {\frac{1 }{2}}}%
T(X,Y),
\end{equation*}%
where $\left\langle X:Y\right\rangle ^{s}$ is given in \emph{(\ref%
{formula_for_symetric_bracket})}, then $\nabla $ is compatible with the
metric $g$.
\end{corollary}

In particular, we can write the Levi-Civita connection explicitly:

\begin{corollary}
\label{Corollary_L_C_connection_form}The Levi-Civita connection with respect
to $g$ is given by%
\begin{equation*}
\nabla _{X}Y={\textstyle{\frac{1}{2}}}\left( \left[ X,Y\right] +\left\langle
X:Y\right\rangle ^{s}\right) ,
\end{equation*}%
where%
\begin{equation}
\left\langle X:Y\right\rangle ^{s}=\sharp (\mathcal{L}_{X}^{a}Y^{\flat }+%
\mathcal{L}_{Y}^{a}X^{\flat }-d^{a}(g(X,Y))  \label{Bracket_Sym_1}
\end{equation}%
for $X,Y\in \Gamma (A)$.
\end{corollary}

\begin{theorem}
The mapping $\left\{ \cdot ,\cdot \right\} ^{s}:\Gamma (A)\times \Gamma
(A)\rightarrow \Gamma (A)$ defined by%
\begin{equation}
\left\{ X,Y\right\} ^{s}=\sharp (\mathcal{L}_{X}^{s}Y^{\flat }+\mathcal{L}%
_{Y}^{s}X^{\flat }+d^{s}(g(X,Y))  \label{Bracket_Sym_2}
\end{equation}%
for $X,Y\in \Gamma (A)$ is a symmetric bracket in the skew-symmetric
algebroid $\left( A,\varrho _{A},[\cdot ,\cdot ]\right) $.
\end{theorem}

\begin{proof}
It is obvious that $\left\{ \cdot ,\cdot \right\} ^{s}$ is a symmetric
bilinear mapping over $\mathbb{R}$. Let $X,Y\in \Gamma (A)$ and $f\in
C^{\infty }(M)$. On account of the properties of the symmetric Lie
derivatives written in Lemma \ref{Lemma_Properties_of_Lie_Sym},%
\begin{equation*}
\mathcal{L}_{X}^{s}(fY)^{\flat }=f\mathcal{L}_{X}^{s}(Y^{\flat })+(\varrho
_{A}\circ X)(f)Y^{\flat }
\end{equation*}%
and%
\begin{equation*}
\mathcal{L}_{fY}^{s}X^{\flat }=f\mathcal{L}_{Y}^{s}X^{\flat }-(i_{Y}X^{\flat
})d^{s}f=f\mathcal{L}_{Y}^{s}X^{\flat }-g(X,Y)d^{s}f.
\end{equation*}%
Furthermore, since%
\begin{equation*}
d^{s}(g(X,fY))=fd^{s}(g(X,Y))+g(X,Y)d^{s}f,
\end{equation*}%
we immediately conclude that%
\begin{equation*}
\left\{ X,fY\right\} ^{s}=f\left\{ X,Y\right\} ^{s}+\sharp ((\varrho
_{A}\circ X)(f)Y^{\flat })=f\left\{ X,Y\right\} +(\varrho _{A}\circ X)(f)Y,
\end{equation*}%
and thus, the proof is complete.
\end{proof}

To compare the symmetric brackets $\left\langle \cdot :\cdot \right\rangle
^{s}$ and $\left\{ \cdot ,\cdot \right\} ^{s}$ given in (\ref{Bracket_Sym_1}%
) and (\ref{Bracket_Sym_2}), respectively, we note that $\left\langle \cdot
:\cdot \right\rangle ^{s}$ is a symmetric product induced by the Levi-Civita
connection, and then, for any $X,Y\in \Gamma (A)$, we have%
\begin{eqnarray*}
\mathcal{L}_{X}^{s}Y^{\flat } &=&\mathcal{L}_{X}^{s}i_{Y}g=i_{Y}\mathcal{L}%
_{X}^{s}g+i_{\left\langle X:Y\right\rangle }g=-i_{Y}\mathcal{L}%
_{X}^{a}g+i_{\left\langle X:Y\right\rangle }g \\
&=&-\mathcal{L}_{X}^{a}Y^{\flat }+i_{\left[ X,Y\right] }g+i_{\left\langle
X:Y\right\rangle }g
\end{eqnarray*}%
since Theorem \ref{Theorem_Comptibility_by_Lie_der} and the Cartan
identities for Lie derivatives given in lemmas \ref{Alternating_Cartan} and %
\ref{Symmetric_Cartan} hold. It follows that%
\begin{equation*}
\left\{ X,Y\right\} ^{s}=2\left\langle X:Y\right\rangle -\left\langle
X:Y\right\rangle ^{s}
\end{equation*}%
for $X,Y\in \Gamma (A)$. Note that there is a more general property saying
that the affine sum of symmetric brackets is again a symmetric bracket.

\section{Symmetric brackets on almost Hermitian manifolds}

In this section we consider the symmetric brackets induced by the structures
of almost Hermitian manifolds. Let $\left( M,g,J\right) $ be an almost
Hermitian manifold, i.e., $\left( M,g\right) $ is a $2n$-dimensional
Riemannian manifold admitting an orthogonal almost complex structure $%
J:TM\rightarrow TM$. Associated to the structures $g$ and $J$ are the \emph{K%
\"{a}hler form} $\Omega \in \Gamma (\mathop{\textstyle \bigwedge }\nolimits%
^{2}T^{\ast }M)$ given by 
\begin{equation*}
\Omega (X,Y)=g(JX,Y)
\end{equation*}%
for $X,Y\in \Gamma (TM)$ and the \emph{Nijenhuis tensor} $N_{J}\in \Gamma (%
\mathop{\textstyle \bigwedge }\nolimits^{2}T^{\ast }M\otimes TM)$ of $J$,
which is defined by%
\begin{equation*}
N_{J}(X,Y)=J\left[ JX,Y\right] +J\left[ X,JY\right] +\left[ X,Y\right] -%
\left[ JX,JY\right] 
\end{equation*}%
for $X,Y\in \Gamma (TM)$.

Kosmann-Schwarzbach and Magri introduced in \cite%
{Kosmann-Schwarzbach-Magri1990} (cf. also \cite{Grabowski-Urbanski1997}) the
bracket $[\![\cdot ,\cdot ]\!]^{J}$ on $TM$ defined by 
\begin{equation}
\lbrack \![X,Y]\!]^{J}=\left[ JX,Y\right] +\left[ X,JY\right] -J\left[ X,Y%
\right] .  \label{Nawias_skew_J}
\end{equation}%
One can observe that for any $X,Y\in \Gamma (TM)$ we have%
\begin{equation*}
N_{J}(X,Y)=J[\![X,Y]\!]^{J}-\left[ JX,JY\right] .
\end{equation*}%
Since%
\begin{equation*}
\lbrack \![X,fY]\!]^{J}=f[\![X,Y]\!]^{J}+(JX)(f)Y
\end{equation*}%
for $X,Y\in \Gamma (TM)$ and $f\in C^{\infty }(M)$, the tangent bundle
together with the almost complex structure $J$ as an anchor and the mapping $%
[\![\cdot ,\cdot ]\!]^{J}$ given in (\ref{Nawias_skew_J}) as a
skew-symmetric bracket is a skew-symmetric algebroid, which we denote by $%
TM^{J}$. It is obvious that if $N_{J}=0$, then $[\![[\![X,Y]\!]^{J},Z]%
\!]^{J}=-J\left[ \left[ JX,JY\right] ,JZ\right] $ for $X,Y,Z\in \Gamma (TM)$
and so $\mathop{\rm Jac}\nolimits_{[\![\cdot ,\cdot ]\!]^{J}}(X,Y,Z)=-J%
\mathop{\rm Jac}\nolimits_{[\cdot ,\cdot ]}(JX,JY,JZ)=0$ for $X,Y,Z\in
\Gamma (TM)$. In consequence, if the almost complex structure $J$ is
integrable, then the skew-symmetric algebroid $\left( TM,J,[\![\cdot ,\cdot
]\!]^{J}\right) $ is a Lie algebroid over $M$.

Let $d^{a}$ and $\mathcal{L}^{a}$ be the exterior derivative operator and
the (alternating) Lie derivative for the Lie algebroid $\left( TM,%
\mathop{\rm Id}\nolimits_{TM},\left[ \cdot ,\cdot \right] \right) $,
respectively, where $\left[ \cdot ,\cdot \right] $ is the Lie bracket of
vector fields on $M$. However, let $d^{J}$ and $\mathcal{L}^{J}$ be the
exterior derivative operator and the Lie derivative for the skew-symmetric
algebroid $\left( TM,J,[\![\cdot ,\cdot ]\!]^{J}\right) $, respectively.

Using the form of symmetric brackets induced by connections with totally
skew-symmetric torsion, in particular by the Levi-Civita connections
(Corollary \ref{Corollary_L_C_connection_form}), we compare symmetric
brackets induced by the Levi-Civita connections in skew-symmetric algebroids 
$TM$ and $TM^{J}$, obtaining the following theorem.

\begin{theorem}
Let $\nabla ^{g}$ be the Levi-Civita connection in the Lie algebroid $\left(
TM,\mathop{\rm Id}\nolimits_{TM},\left[ \cdot ,\cdot \right] \right) $ and
let $\nabla ^{J,g}$ be the Levi-Civita connection in the skew-symmetric
algebroid $\left( TM,J,[\![\cdot ,\cdot ]\!]^{J}\right) $, both metric with
respect to $g$. If $\left\langle X:Y\right\rangle =\nabla _{X}^{g}Y+\nabla
_{Y}^{g}X$ and $\left\langle X:Y\right\rangle ^{J}=\nabla _{X}^{J,g}Y+\nabla
_{Y}^{J,g}X $ for $X,Y\in \Gamma (TM)$, then 
\begin{equation}
\left\langle X:Y\right\rangle ^{J}=\left\langle JX:Y\right\rangle
+\left\langle X:JY\right\rangle +\sharp \hspace{-0.1cm}\left( \left\langle
X:Y\right\rangle ^{\flat }\circ J\right) .  \label{Formula_Sym_bracket_forJ}
\end{equation}
\end{theorem}

\begin{proof}
Let $X,Y\in \Gamma (TM)$. Corollary \ref{Corollary_L_C_connection_form} now
yields%
\begin{equation*}
\left\langle X:Y\right\rangle ^{J}=\sharp (\mathcal{L}_{X}^{J}Y^{\flat }+%
\mathcal{L}_{Y}^{J}X^{\flat }-d^{J}(g(X,Y)).
\end{equation*}%
One can check that%
\begin{equation*}
\mathcal{L}_{X}^{J}Y^{\flat }=\mathcal{L}_{JX}^{a}Y^{\flat }+\mathcal{L}%
_{X}^{a}(JY)^{\flat }+(\mathcal{L}_{X}^{a}Y^{\flat })\circ J.
\end{equation*}%
Moreover, since%
\begin{equation*}
d^{J}(g(X,Y))=d^{a}(g(X,Y))\circ J
\end{equation*}%
and%
\begin{equation*}
d^{a}(g(JX,Y))+d^{a}(g(X,JY))=0,
\end{equation*}%
we have (\ref{Formula_Sym_bracket_forJ}).
\end{proof}

Now, we define some symmetric brackets on almost Hermitian manifolds.

Let $\left( TM,\rho ,\left[ \cdot ,\cdot \right] ^{\rho }\right) $ be a
structure of skew-symmetric algebroid, and let $\left\langle \cdot :\cdot
\right\rangle ^{\rho }$ be a symmetric bracket in this algebroid. By
definition, 
\begin{equation*}
\left\langle X:fY\right\rangle ^{\rho }=f\left\langle X:Y\right\rangle
^{\rho }+(\rho \circ X)(f)Y
\end{equation*}%
for $X,Y\in \Gamma (TM)$.

We define two $\mathbb{R} $-bilinear symmetric operators $P^{^{\rho
}},Q^{^{\rho }}:\Gamma (TM)\times \Gamma (TM)\rightarrow \Gamma (TM)$,%
\begin{equation*}
P^{^{\rho }}(X,Y)=-J\left( \left[ X,JY\right] ^{\rho }+\left[ Y,JX\right]
^{\rho }\right)
\end{equation*}%
and%
\begin{equation*}
Q^{^{\rho }}(X,Y)=-J\left( \left\langle X:JY\right\rangle ^{\rho
}+\left\langle Y:JX\right\rangle ^{\rho }\right)
\end{equation*}%
for $X,Y\in \Gamma (TM)$.

\begin{lemma}
\label{Properties_of_P_and_Q}For any $X,Y\in \Gamma (TM)$, $f\in C^{\infty
}(M)$, we have

\begin{itemize}
\item[(a)] $P^{^{\rho }}(X,f\cdot Y)=f\cdot P^{^{\rho }}(X,Y)+(\rho \circ
X)(f)\cdot Y+(\rho \circ JX)(f)\cdot JY$ and

\item[(b)] $Q^{^{\rho }}(X,f\cdot Y)=f\cdot Q^{^{\rho }}(X,Y)+(\rho \circ
X)(f)\cdot Y-(\rho \circ JX)(f)\cdot JY.$
\end{itemize}
\end{lemma}

\begin{proof}
Compute directly,%
\begin{eqnarray*}
P^{^{\rho }}(X,f\cdot Y) &=&-J\left( \left[ X,f\cdot JY\right] ^{\rho }+%
\left[ f\cdot Y,JX\right] ^{\rho }\right) \\
&=&-J\left( f\cdot \left[ X,JY\right] ^{\rho }+(\rho \circ X)(f)\cdot
JY+f\cdot \left[ Y,JX\right] ^{\rho }-(\rho \circ JX)(f)\cdot Y\right) \\
&=&f\cdot P^{^{\rho }}(X,Y)+(\rho \circ X)(f)\cdot Y+(\rho \circ JX)(f)\cdot
JY
\end{eqnarray*}%
and 
\begin{eqnarray*}
Q^{^{\rho }}(X,f\cdot Y) &=&-J\left( \left\langle X:f\cdot JY\right\rangle
^{\rho }+\left\langle f\cdot Y:JX\right\rangle ^{\rho }\right) \\
&=&-J\left( f\cdot \left\langle X:JY\right\rangle ^{\rho }+(\rho \circ
X)(f)\cdot JY+f\cdot \left\langle Y:JX\right\rangle ^{\rho }+(\rho \circ
JX)(f)\cdot Y\right) \\
&=&f\cdot Q^{^{\rho }}(X,Y)+(\rho \circ X)(f)\cdot Y-(\rho \circ JX)(f)\cdot
JY.
\end{eqnarray*}
\end{proof}

In consequence of Lemma \ref{Properties_of_P_and_Q}, we immediately get the
following results.

\begin{theorem}
\label{Theorem_P_and_Q}The mapping%
\begin{equation*}
{\textstyle {\frac{1 }{2}}}(P^{^{\rho }}+Q^{^{\rho }})
\end{equation*}%
is a symmetric bracket in the skew-symmetric algebroid $\left( TM,\rho ,%
\left[ \cdot ,\cdot \right] ^{\rho }\right) $.
\end{theorem}

\begin{corollary}
\label{Lemma_formula_for_sym_bracketPQ}The mapping $\left\langle \cdot
:\cdot \right\rangle :\Gamma (TM)\times \Gamma (TM)\rightarrow \Gamma (TM)$
given by%
\begin{equation*}
\left\langle X:Y\right\rangle =-{\textstyle{\frac{1}{2}}}J\left( \left[ X,JY%
\right] +\left[ Y,JX\right] +\sharp (\mathcal{L}_{X}^{a}(JY)^{\flat }+%
\mathcal{L}_{Y}^{a}(JX)^{\flat }+\mathcal{L}_{JX}^{a}Y^{\flat }+\mathcal{L}%
_{JY}^{a}X^{\flat })\right)
\end{equation*}%
is a symmetric bracket in the Lie algebroid $\left( TM,\mathop{\rm Id}%
\nolimits_{TM},\left[ \cdot ,\cdot \right] \right) $, where $\left[ \cdot
,\cdot \right] $ is the Lie bracket of vector fields on $M$ and $\mathcal{L}%
^{a}$ is the Lie derivative on $M$.
\end{corollary}

\begin{proof}
Let $d^{a}$ be the exterior derivative on manifold $M$. Taking $\rho =%
\mathop{\rm Id}\nolimits_{TM}$ in Theorem \ref{Theorem_P_and_Q} and using
Theorem \ref{Theorem_Sym_Bracket_by_Lie_der}, we deduce that the formula%
\begin{eqnarray*}
\left\langle X:Y\right\rangle  &=&-{\textstyle{\frac{1}{2}}}J\left( \left[
X,JY\right] +\left[ Y,JX\right] \right) -{\textstyle{\frac{1}{2}}}(J\circ
\sharp )(\mathcal{L}_{X}^{a}(JY)^{\flat }+\mathcal{L}_{JY}^{a}X^{\flat
}-d^{a}(g(X,JY)) \\
&&-{\textstyle{\frac{1}{2}}}(J\circ \sharp )(\mathcal{L}_{JX}^{a}Y^{\flat }+%
\mathcal{L}_{Y}^{a}(JX)^{\flat }-d^{a}(g(JX,Y))
\end{eqnarray*}%
defines a symmetric bracket in the tangent bundle with $\mathop{\rm Id}%
\nolimits_{TM}$ as an anchor and with the classical Lie bracket. Since $%
\Omega $ is a skew-symmetric $2$-form on $M$, it follows that%
\begin{equation*}
g(X,JY)+g(JX,Y)=\Omega (Y,X)+\Omega (Y,X)=0.
\end{equation*}%
Therefore,%
\begin{eqnarray*}
\left\langle X:Y\right\rangle  &=&-{\textstyle{\frac{1}{2}}}J\left( \left[
X,JY\right] +\left[ Y,JX\right] \right)  \\
&&-{\textstyle{\frac{1}{2}}}(J\circ \sharp )(\mathcal{L}_{X}^{a}(JY)^{\flat
}+\mathcal{L}_{Y}^{a}(JX)^{\flat }+\mathcal{L}_{JX}^{a}Y^{\flat }+\mathcal{L}%
_{JY}^{a}X^{\flat }).
\end{eqnarray*}
\end{proof}

It is obvious that the bracket in Corollary \ref%
{Lemma_formula_for_sym_bracketPQ} is a totally symmetric part of the
connection $\nabla ^{J}:\Gamma (TM)\times \Gamma (TM)\rightarrow \Gamma (TM)$
defined by 
\begin{equation*}
\nabla _{X}^{J}Y=-{\textstyle{\frac{1}{2}}}J\left( \left[ X,JY\right]
+\left\langle X:JY\right\rangle \right) .
\end{equation*}%
Let $\nabla $ be the Levi-Civita connection with respect to $g$ given by%
\begin{equation*}
\nabla _{X}Y={\textstyle{\frac{1}{2}}}\left( \left[ X,Y\right] +\left\langle
X:Y\right\rangle ^{\nabla }\right) .
\end{equation*}%
Now, let $\left\langle \cdot :\cdot \cdot \right\rangle =\left\langle \cdot
:\cdot \cdot \right\rangle ^{\nabla }$. Hence, 
\begin{equation*}
\nabla _{X}^{J}Y=-J\nabla _{X}(JY).
\end{equation*}%
One can observe that the affine sum 
\begin{equation*}
\overline{\nabla }={\textstyle{\frac{1}{2}}}\left( \nabla +\nabla
^{J}\right) 
\end{equation*}%
of connections $\nabla $ and $\nabla ^{J}$ is Lichnerowicz's first canonical
Hermitian connection \cite{Lichnerowicz1962}, which is compatible with both
the metric structure and the almost complex structure. This is a direct
consequence of the properties of $\nabla $ and $\nabla ^{J}$ given in the
following lemma.

\begin{lemma}
\label{Lemma_Prop_Nabla_Bar_01}~

\begin{itemize}
\item[(a)] $(\nabla ^{J}g)(X,Y,Z)=(\nabla g)(X,JY,JZ)$ for $X,Y,Z\in \Gamma
(TM)$ and

\item[(b)] $\nabla ^{J}J=-\nabla J$.
\end{itemize}
\end{lemma}

Lemma \ref{Lemma_Prop_Nabla_Bar_01} now yields%
\begin{equation*}
\nabla ^{J}g=0,\ \ \overline{\nabla }={\textstyle {\frac{1 }{2}}}\left(
\nabla +\nabla ^{J}\right) g={\textstyle {\frac{1 }{2}}}\left( \nabla
g+\nabla ^{J}g\right) =0
\end{equation*}%
and%
\begin{equation*}
\overline{\nabla }J={\textstyle {\frac{1 }{2}}}\left( \nabla J+\nabla
^{J}J\right) ={\textstyle {\frac{1 }{2}}}\left( \nabla J-\nabla J\right) =0.
\end{equation*}%
We will now consider some further properties of $\nabla ^{J}$ and $\overline{%
\nabla }$.

For an $A$-connection $\nabla $ on $A$, we define the operators%
\begin{equation*}
d_{\nabla }^{a},d_{\nabla }^{s}:\Gamma (\mathop{\textstyle \bigotimes }%
\nolimits^{k}T^{\ast }M)\rightarrow \Gamma (\mathop{\textstyle \bigotimes }%
\nolimits^{k+1}T^{\ast }M)
\end{equation*}%
as the alternation and the symmetrization of $\nabla $, respectively, i.e.,
for $\zeta \in \Gamma (\mathop{\textstyle \bigotimes }\nolimits^{k}T^{\ast
}M)$, $X_{1},\ldots ,X_{k+1}\in \Gamma (TM)$, we have%
\begin{equation*}
(d_{\nabla }^{a}\zeta )(X_{1},\ldots ,X_{k+1})=\mathop{\textstyle \sum }%
\limits_{i=1}^{k+1}\left( -1\right) ^{i+1}\left( \nabla _{X_{i}}\zeta
\right) (X_{1},\ldots \widehat{X}_{i}\ldots ,X_{k+1})
\end{equation*}%
and%
\begin{equation*}
(d_{\nabla }^{s}\zeta )(X_{1},\ldots ,X_{k+1})=\mathop{\textstyle \sum }%
\limits_{i=1}^{k+1}\left( \nabla _{X_{i}}\zeta \right) (X_{1},\ldots 
\widehat{X}_{i}\ldots ,X_{k+1}).
\end{equation*}

We say that an almost Hermitian manifold $\left( M,g,J\right) $ is \emph{%
nearly K\"{a}hler} if $\left( \nabla _{X}J\right) Y=-\left( \nabla
_{Y}J\right) Y$ for $X,Y\in \Gamma (TM)$ (cf. \cite{Gray1970}). We have the
following lemma.

\begin{lemma}
\label{Lemma_on_nearly_Kahler}An almost manifold $\left( M,g,J\right) $ is
nearly K\"{a}hler if and only if $d_{\nabla }^{s}J=0$.
\end{lemma}

Moreover, if $\left( M,g,J\right) $ is nearly K\"{a}hler, $\overline{\nabla }
$ is a Hermitian connection with totally skew-symmetric torsion (e.g., cf. 
\cite{Agricola-Srni}).

Now, we compare the symmetric brackets induced by $\nabla $ and $\overline{%
\nabla }$. We will denote by $\left\langle \cdot :\cdot \right\rangle ^{%
\overline{\nabla }}$ the symmetric product of $\overline{\nabla }$.

\begin{theorem}
\label{Theorem_SymNabla_and_Nabla_J}For $X,Y\in \Gamma (TM)$, we have%
\begin{equation*}
J((d_{\nabla }^{s}J)(X,Y))=\left\langle X:Y\right\rangle ^{\nabla
}-\left\langle X:Y\right\rangle ^{\nabla ^{J}}.
\end{equation*}
\end{theorem}

\begin{proof}
We first observe that%
\begin{eqnarray*}
(d_{\nabla }^{s}J)(X,Y) &=&(\nabla _{X}J)Y+(\nabla _{Y}J)X \\
&=&\nabla _{X}(JY)+\nabla _{Y}(JX)-J(\nabla _{X}Y+\nabla _{Y}X) \\
&=&\nabla _{X}(JY)+\nabla _{Y}(JX)-J\left\langle X:Y\right\rangle ^{\nabla }.
\end{eqnarray*}%
From this equality, we obtain%
\begin{eqnarray*}
J((d_{\nabla }^{s}J)(X,Y)) &=&J\nabla _{X}(JY)+J\nabla _{Y}(JX)+\left\langle
X:Y\right\rangle ^{\nabla } \\
&=&-\left\langle X:Y\right\rangle ^{\nabla ^{J}}+\left\langle
X:Y\right\rangle ^{\nabla }.
\end{eqnarray*}
\end{proof}

\begin{theorem}
\label{Theorem_Bracket_Nabla_Bar}For $X,Y\in \Gamma (TM)$, we have%
\begin{equation*}
\left\langle X:Y\right\rangle ^{\overline{\nabla }}=\left\langle
X:Y\right\rangle ^{\nabla }-{\textstyle {\frac{1 }{2}}}J((d_{\nabla
}^{s}J)(X,Y)).
\end{equation*}
\end{theorem}

\begin{proof}
Since $\overline{\nabla }={\textstyle{\frac{1}{2}}}\left( \nabla +\nabla
^{J}\right) $ is an affine sum of connections $\nabla $ and $\nabla ^{J}$,%
\begin{equation*}
\left\langle X:Y\right\rangle ^{\overline{\nabla }}={\textstyle{\frac{1}{2}}}%
\left\langle X:Y\right\rangle ^{\nabla }+{\textstyle{\frac{1}{2}}}%
\left\langle X:Y\right\rangle ^{\nabla ^{J}}.
\end{equation*}%
From this result and Theorem \ref{Theorem_SymNabla_and_Nabla_J}, we see that%
\begin{eqnarray*}
\left\langle X:Y\right\rangle ^{\overline{\nabla }} &=&{\textstyle{\frac{1}{2%
}}}\left\langle X:Y\right\rangle ^{\nabla }+{\textstyle{\frac{1}{2}}}\left(
\left\langle X:Y\right\rangle ^{\nabla }-J((d_{\nabla }^{s}J)(X,Y))\right) 
\\
&=&\left\langle X:Y\right\rangle ^{\nabla }-{\textstyle{\frac{1}{2}}}%
J((d_{\nabla }^{s}J)(X,Y)).
\end{eqnarray*}
\end{proof}

Since $\overline{\nabla }={\textstyle {\frac{1 }{2}}}\left( \nabla +\nabla
^{J}\right) $ and $\nabla $ is torsion-free, we have 
\begin{equation*}
T^{\overline{\nabla }}={\textstyle {\frac{1 }{2}}}T^{\nabla }+{\textstyle {%
\frac{1 }{2}}}T^{\nabla ^{J}}={\textstyle {\frac{1 }{2}}}T^{\nabla ^{J}}.
\end{equation*}

\begin{theorem}
\label{Torsion_of_Nabla_J}$T^{\nabla ^{J}}=-J\circ \left( d_{\nabla
}^{a}J\right) $.
\end{theorem}

\begin{proof}
Let $X,Y\in \Gamma (TM)$. Then%
\begin{eqnarray*}
(d_{\nabla }^{a}J)(X,Y) &=&(\nabla _{X}J)Y-(\nabla _{Y}J)X \\
&=&\nabla _{X}(JY)-\nabla _{Y}(JX)-J\left[ X,Y\right] .
\end{eqnarray*}%
Hence,%
\begin{eqnarray*}
-J((d_{\nabla }^{a}J)(X,Y)) &=&-J\nabla _{X}(JY)-(-J\nabla _{Y}(JX))+J^{2} 
\left[ X,Y\right] \\
&=&\nabla _{X}^{J}Y-\nabla _{Y}^{J}X-\left[ X,Y\right] =T^{\nabla ^{J}}(X,Y).
\end{eqnarray*}
\end{proof}

\begin{theorem}
\label{Torsion_of_Nabla_J_second}For $X,Y\in \Gamma (TM)$, we have%
\begin{equation*}
2T^{\nabla ^{J}}(X,Y)=-N_{J}(X,Y)+(d_{\nabla }^{s}J)(X,JY)-(d_{\nabla
}^{s}J)(JX,Y).
\end{equation*}%
In particular, if $\left( M,g,J\right) $ is nearly K\"{a}hler, then%
\begin{equation*}
T^{\nabla ^{J}}=-{\textstyle {\frac{1 }{2}}}N_{J}.
\end{equation*}
\end{theorem}

\begin{proof}
Let $X,Y\in \Gamma (TM)$. Then (e.g., \cite{Agricola-Srni} shows the first
equality):%
\begin{eqnarray*}
-N_{J}(X,Y) &=&(\nabla _{X}J)JY-(\nabla _{Y}J)JX+(\nabla _{JX}J)Y-(\nabla
_{JY}J)X \\
&=&(\nabla _{X}J)JY-(\nabla _{Y}J)JX-(\nabla _{Y}J)JX+(d_{\nabla
}^{s}J)(JX,Y) \\
&&+(\nabla _{X}J)JY-(d_{\nabla }^{s}J)(X,JY) \\
&=&2\left( (\nabla _{X}J)JY-(\nabla _{Y}J)JX\right) +(d_{\nabla
}^{s}J)(JX,Y)-(d_{\nabla }^{s}J)(X,JY).
\end{eqnarray*}%
Moreover,%
\begin{eqnarray*}
(\nabla _{X}J)JY-(\nabla _{Y}J)JX &=&-\nabla _{X}Y-J(\nabla _{X}(JY))+\nabla
_{Y}X+J(\nabla _{Y}(JX)) \\
&=&-J(\nabla _{X}(JY))-(-J(\nabla _{Y}(JX)))-\nabla _{X}Y+\nabla _{Y}X \\
&=&\nabla _{X}^{J}Y-\nabla _{Y}^{J}X-\left[ X,Y\right] =T^{\nabla ^{J}}(X,Y).
\end{eqnarray*}%
It follows that%
\begin{equation*}
-N_{J}(X,Y)=2T^{\nabla ^{J}}(X,Y)+(d_{\nabla }^{s}J)(JX,Y)-(d_{\nabla
}^{s}J)(X,JY).
\end{equation*}
\end{proof}

Since $\overline{\nabla }$ is a totally skew-symmetric connection, Theorem %
\ref{Theorem_Bracket_Nabla_Bar} now leads to 
\begin{eqnarray}
\overline{\nabla }_{X}Y &=&{\textstyle {\frac{1 }{2}}}\left( \left[ X,Y%
\right] +\left\langle X:Y\right\rangle ^{\overline{\nabla }}\right) +{%
\textstyle {\frac{1 }{2}}}T^{\overline{\nabla }}(X,Y)
\label{formula_nearly_final_for_Nabla_Bar} \\
&=&{\textstyle {\frac{1 }{2}}}\left( \left[ X,Y\right] +\left\langle
X:Y\right\rangle ^{\nabla }-J((d_{\nabla }^{s}J)(X,Y))\right) +{\textstyle {%
\frac{1 }{2}}}T^{\overline{\nabla }}(X,Y)  \notag \\
&=&\nabla _{X}Y-{\textstyle {\frac{1 }{2}}}J((d_{\nabla }^{s}J)(X,Y))+{%
\textstyle {\frac{1 }{4}}}T^{\nabla ^{J}}(X,Y).  \notag
\end{eqnarray}

Combining (\ref{formula_nearly_final_for_Nabla_Bar}) with Lemma \ref%
{Lemma_on_nearly_Kahler} and Theorems \ref{Torsion_of_Nabla_J} and \ref%
{Torsion_of_Nabla_J_second}, we get the following result:

\begin{corollary}
If $\left( M,g,J\right) $ is nearly K\"{a}hler, then $d_{\nabla }^{s}J=0$,
and, in consequence, 
\begin{equation*}
\overline{\nabla }=\nabla -{\textstyle {\frac{1 }{4}}}J\circ (d_{\nabla
}^{a}J)=\nabla -{\textstyle {\frac{1 }{8}}}N_{J}.
\end{equation*}
\end{corollary}

\

\end{document}